\documentclass[a4paper,12pt]{article}
\usepackage[margin=1in]{geometry}

\usepackage{amsmath}
\usepackage{amssymb}
\usepackage{amscd}
\usepackage{amsthm}
\usepackage{mathabx}

\usepackage{enumerate}
\usepackage{tikz}
\usetikzlibrary{cd}
\usepackage{geometry,hyperref}
\newtheorem{theorem}{Theorem}
\newtheorem{lemma}{Lemma}

\newtheorem{remark}{Remark}
\newtheorem{result}{Result}
\newtheorem{corollary}{Corollary}
\newtheorem{example}{Example}
\newcommand{\cV}{\mathcal V}

\newcommand{\cS}{\mathcal S}

\newcommand{\fF}{\mathfrak F}

\newcommand{\PG}{\mathrm{PG}\,}
\newcommand{\AG}{{\mathrm{AG}}\,}
\newcommand{\PGL}{{\mathrm{PGL}}\,}

\newcommand{\FF}{{\mathbb F}}

\newtheorem{proposition}[theorem]{Proposition}

\def\eqn#1$$#2$${\begin{equation}\label#1#2\end{equation}}

\def\F{\mathbb F}

\def\cS{{\cal S}}

\def\cV{{\cal V}}

\def\la{\langle}
\def\ra{\rangle}
\DeclareMathOperator*{\opl}{\oplus}
\begin{document}
	\title{On the maximum field of linearity of linear sets}
	\author{Bence Csajb\'ok,
		Giuseppe Marino, Valentina Pepe}
	
	\maketitle

	\begin{abstract}

		Let $V$ denote an $r$-dimensional $\F_{q^n}$-vector space.
		For an $m$-dimensional $\F_q$-subspace $U$ of $V$ assume that $\dim_q \left(\la {\bf v}\ra_{\F_{q^n}} \cap U\right) \geq 2$ for each non zero vector ${\bf v}\in U$.
		If $n\leq q$ then we prove the existence of an integer $1<d \mid n$ such that the set of one-dimensional $\F_{q^n}$-subspaces generated by non-zero vectors of $U$ is the same as the set of one-dimensional
		$\F_{q^n}$-subspaces generated by non-zero vectors of $\la U\ra_{\F_{q^d}}$. If we view $U$ as a point set of $\AG(r,q^n)$, it means that $U$ and $\la U \ra_{\F_{q^d}}$ determine the same set of directions. We prove a stronger statement when $n \mid m$.
		
		In terms of linear sets it means that an $\F_q$-linear set of $\PG(r-1,q^n)$ has maximum field of linearity $\F_q$ only if it has a point of weight one.
		We also present some consequences regarding the size of a linear set.
	\end{abstract}

	\section{Introduction}

	Let $V$ be an $r$-dimensional vector space over $\F_{q^n}$, $q$ is a power of the prime $p$, and let $\PG(V,\F_{q^n})$ be the associated projective space.
	A point set $L$ of $\PG(V,\F_{q^n})=\PG(r-1,q^n)$ is called an \emph{$\F_q$-linear set} of {\it rank} $m$ if it is
	defined by the non-zero vectors of a $m$-dimensional $\F_q$-vector subspace $U$ of $V$, i.e.
	\[L=L_U=\{\la {\bf u} \ra_{\F_{q^n}} \colon {\bf u}\in U\setminus \{{\bf 0} \}\}.\]
	We point out that different vector subspaces can define the same linear set.
	If $r=m$ and $\la L_U \ra=\PG(r-1,q^n)$, then $L_U$ is called a {\it $q$-order subgeometry} of $\PG(r-1,q^n)$.
	
	By Grassmann's Identity it can be easily seen that if $m>(r-1)n$ then $L_U=\PG(V,\F_{q^n})$. Hence  $L_U$ is a proper subset of $\PG(V,\F_{q^n})$ if and only if $\dim \PG(U,\FF_q)\leq rn-n-1$, i.e. the rank of the linear set is at most $rn-n$. A linear set of rank $rn-n$ is said to be of \textit{maximum rank}. Hence, in the rest of the paper, we will assume $ m \leq rn-n$.
	
	\medskip
	For a point $P=\la {\bf v} \ra_{\F_{q^n}}\in \PG(V,\F_{q^n})$ the \emph{weight} of $P$ with respect to the linear set $L_U$ is $w_{L_U}(P):=\dim_q(\la {\bf v} \ra_{\F_{q^n}} \cap U)$.
	
	\medskip
	
	The \emph{maximum field of linearity} of an $\F_q$-linear set $L$ is $\F_{q^t}$ if $t \mid n$ is the largest integer such that $L=L_U$ for some $\F_{q^t}$-subspace $U$, cf. \cite[pg.\ 403]{CsMP}.
	
	\begin{remark}
		Note that the maximum field of linearity $\F_{q^t}$ of the $\F_q$-linear set $L_U$ of $\PG(V,\F_{q^n})$ is always an extension of the field $\F_q$. However, we cannot exclude the existence of a subspace $W$ of $V$ over a subfield $\F_{p^s}$ of $\F_{q^n}$, $p^s > q^t$, which is not an extension of $\F_q$ but it satisfies $L_U=L_W$.
	\end{remark}
	
	In the recent years, starting from the paper \cite{Lu1999} by Lunardon, linear sets have been used to construct or characterize various objects in finite geometry, such as blocking sets and multiple blocking sets in finite projective spaces, two-intersection sets in finite projective spaces, translation spreads of the Cayley Generalized Hexagon, translation ovoids of polar spaces, semifield flocks and finite semifields. For a survey on linear sets we refer the reader to \cite{OP2010}.
	
	\medskip
	
	First recall what is known about the maximum field of linearity of linear sets of the projective line $\PG(1,q^n)$.
	Consider the $\F_q$-linear set $L_U$ of $\PG(1,q^n)$. If $\dim_q\,U > n$ then $L_U=\PG(1,q^n)$. If $\dim_q\,U = n$ then, without loss of generality we may assume that $\langle (0,1)\rangle_{\F_{q^n}} \notin L_U$ and hence
	\[U=U_f:=\{(x,f(x))\colon x\in \F_{q^n}\},\]
	where $f(x)$ is a $q$-polynomial of the form $\sum_{i=0}^{n-1} a_i x^{q^i}\in \F_{q^n}[x]$.
	Then the linear set defined by $U$ is
	\[L_f=L_{U}=\{\la (1,f(x)/x) \ra_{\F_{q^n}} : x \in \F_{q^n}\setminus \{0\}\}.\]
	One might also consider $U_f$ as an affine point set in $\AG(2,q^n)$, the graph of the function $f$.
	The points of the line at infinity $\ell_{\infty}$ of $\AG(2,q^n)$ are called {\it directions}. A direction $(d)$ is the common point of the lines with slope $d$.
	The set of directions determined by $f$ is:
	\[D_f:=\left\{\left(\frac{f(x)-f(y)}{x-y}\right)\colon x,y\in \F_{q^n},\, x\neq y\right\}.\]
	Since $f$ is $\F_q$-linear, for each direction $(d)$, there is a non-negative integer $e$,
	such that each line of $\PG(2,q^n)$ with slope $d$ meets $U_f$ in $q^e$ or $0$ points.
	Also, the additivity of $f$ yields that the set of directions determined by $f$ corresponds to the points
	of the linear set $L_f$ and the weight of the point $\la(1,d)\ra_{\F_{q^n}}$ is $e$. Next recall the following important result:

	\begin{result}[Ball et al. \cite{BBBSSz} and Ball \cite{B2003}]
		\label{part}
		Let $f$ be a function from $\F_q$ to $\F_q$, $q=p^h$, and let $N$ be the number of directions determined by $f$.
		Let $s=p^e$ be maximal such that any line with a direction determined by $f$ that is incident with a point of the graph of $f$ is incident with a multiple of $s$ points of the graph of $f$.
		Then one of the following holds.
		\begin{enumerate}
			\item $s=1$ and $(q+3)/2 \leq N \leq q+1$,
			\item $e \mid h$, $q/s + 1 \leq N \leq (q - 1)/(s - 1)$,
			\item $s=q$ and $N=1$.
		\end{enumerate}
		Moreover if $s>2$ and $f(0)=0$ then the graph of $f$ is $\F_s$-linear.
	\end{result}
	
	``$\F_s$-linear'' in the theorem above means that $U_f$ is an $\F_s$-subspace of $\AG(2,q)$, i.e.
	$L_f$ is an $\F_s$-linear set. Hence the result above can be read as follows
	
	\begin{result}\label{CsMP}
		If $d$ is maximal such that each point of the $\F_q$-linear set $L_f$ of rank $n$ of $\PG(1,q^n)$ is of weight at least $d$ then $d \mid n$ and the maximum field of linearity of $L_f$ is $\F_{q^d}$.
	\end{result}
	
	It also follows that if $\dim_q\,U=n$ and $L_U=L_{U'}\subseteq \PG(1,q^n)$ for some $\F_q$-subspace $U'$, then $\dim_q\, U'=n$, i.e. the rank of such linear sets is well-defined, see \cite[Proposition 2.3]{CsMP}.

	\medskip
	Now we introduce our main results.
	
	\medskip
	\noindent
	\begin{theorem} \label{thm:main1}
		For some integer $a$ let $L_U$ be an $\F_q$-linear set of rank $an$ in $\PG(r-1,q^n)$ such that $w$ is maximal with the property that points of $L_U$ have weight at least $w \geq 2$. Then $w \mid n$ and $U$ is $\F_{q^w}$-linear.
	\end{theorem}
	
	In Theorem \ref{d2} we formulate the dual statement in terms of vector spaces. This result is related to $e$-divisible rank metric codes studied in \cite{Olga2022}.

	\begin{theorem}\label{thm:main}
		Let $L_U$ be an $\F_q$-linear set of $\PG(V,\F_{q^n})=\PG(r-1,q^n)$, of rank $m$, $m\leq (r-1)n$, such that the following assumptions are satisfied:
		\begin{itemize}
			\item [$i)$] $n\leq q$;
			\item[$ii)$]
			every point of $L_U$ has weigh at least $w\geq 2$.
		\end{itemize}
		Then there exist an integer $d$ with $w \leq d \mid n$ such that $L_U=L_{U'}$ with $U'=\la U \ra_{\F_{q^d}}$.
	\end{theorem}
	
	In Section \ref{sec:4} we present lower bounds on the size of $\F_q$-linear sets with maximum field of linearity $\F_q$ by using and extending some recent results form \cite{SP,DBVdV}.
	
	\medskip
	In terms of spreads it yields the following. If we have a Desarguesian $(n-1)$-spread $\cS$ of $\PG(rn-1,q)$, $n\leq q$, and a projective subspace $\Gamma$ such that the elements of $\cS$ meet $\Gamma$ in projective
	subspaces of projective dimension at least one, then there exist $1<d \mid n$, a Desarguesian $(d-1)$-spread $\cS'$ of $\PG(rn-1,q)$ and a projective subspace $\Gamma' \supseteq \Gamma$ such that:  $\cS'$ induces a $(d-1)$-spread on each element of $\cS$ and a $(d-1)$-spread on $\Gamma'$, and the subspaces $\Gamma$ and $\Gamma'$ meet the same set of elements of $\cS$.
	
	\medskip
	
	Consider $\mathrm{PG}(r-1,q^n)$ as $\mathrm{AG}(r-1,q^n) \cup H_{\infty}$, where $H_{\infty}$ is the hyperplane at infinity.
	The \emph{set of directions determined} by an affine point set $S$ of $\mathrm{AG}(r-1,q^n)$ is the set of ideal points $\mathrm{dir}(S)=\{\langle P, Q \rangle \cap H_{\infty} : P,Q \in S, P\neq Q \}$.
	By \cite[Theorem 8]{direzioni}, if $U$ is an $m$-dimensional $\F_{q}$-subspace of
	$\AG(r-1,q^n)$ then $\mathrm{dir}(U)$ is a linear set of rank $m$ of $H_{\infty}$ and viceversa, each $\F_q$-linear set of rank $m$ of $H_{\infty}$ can be obtained in this way. In this representation the weight of a point $P\in \mathrm{dir}(U)$ is $w$ exactly when the line joining $P$ with the origin meets $U$ in $q^w$ points. Then our results can be read as follows.
	
	%For a function
	%$f \colon \mathbb{F}_{q^n} \rightarrow \mathbb{F}_{q^n}$ the graph of $f$ is the affine point set
	%$U_f=\{(x,f(x)) : x \in \mathbb{F}_{q^n}\} \subseteq \mathrm{AG}(2,q^n)$.
	%Assume that $f$ (and hence $U_f$) is $\mathbb{F}_{q}$-linear. If $\ell$ is a line through the origin, then $|U_f \cap \ell|=q^{i}$, $i\in \mathbb{Z}_0^+$.
	%Let $w$ be maximal such that for each line $\ell$ through the origin either $U_f\cap \ell=\{(0,0)\}$, or $|U_f \cap \ell|\geq q^{w}$. By a result of Ball et al. $\mathbb{F}_{q^{w}}$ is the largest subfield of $\mathbb{F}_{q^n}$ such that $U_f$ is an $\mathbb{F}_{q^{w}}$-subspace.
	%With G. Marino and V. Pepe we proved the following generalisation.
	
	\begin{theorem}
		Let $U$ denote an $m$-dimensional $\mathbb{F}_q$-subspace of $\mathrm{AG}(r-1,q^n)$.
		Let $w$ be maximal such that for each line $\ell$ through the origin either $U\cap \ell=\{(0,\ldots,0)\}$ or $|U\cap \ell|\geq q^w$.
		\begin{itemize}
			\item[(a)] If $n \divides m$, then $w \mid n$ and $\mathbb{F}_{q^w}$ is the largest subfield of $\mathbb{F}_{q^n}$ such that $U$ is $\mathbb{F}_{q^w}$-linear.
		\end{itemize}
		If $n \notdivides m$ then $U$ is not necessarily linear over a larger field, but it acts similarly:
		\begin{itemize}
			\item[(b)] If $q\geq n$, then there exists an integer $d \geq w$, $d \mid n$, such that $\mathrm{dir}(U)=\mathrm{dir}(\langle U \rangle_{\mathbb{F}_{q^d}})$.
		\end{itemize}
	\end{theorem}
	
	%In terms of $\F_q$-linear functions we can formulate the previous theorem, case $r=2$, as follows.
	
	%\begin{theorem}
	%  For an $\F_q$-subspace $U$ of $\F_{q^n}$, $n\leq q$, and an $\F_q$-linear function $f\colon U \rightarrow \F_{q^n}$ assume that each line meets the graph of $f$ in $0$, $1$, or in at least $q^2$ points.
	% Then there exist $1 < d \mid n$ and an $\F_{q^d}$-linear function $g \colon \la U \ra_{\F_{q^d}} \rightarrow \F_{q^n}$ such that the graphs of $f$ and $g$
	%determine the same set of directions. \qed
	%\end{theorem}

	\section{Proof of Theorem \ref{thm:main1}}
	Let $X_0,X_1,\ldots, X_{r-1}$ be homogenous projective coordinates for $\PG(r-1,q^n)$. By $E_i$, $i=0,1,\ldots,r-1$, we will denote the point $\la(0,\ldots,0,1,0,\ldots,0)\ra_{\F_{q^n}}$ of $\PG(r-1,q^n)$ with $1$ in the $i$-th position and $0$ elsewhere.
	
	\begin{proposition}
		Let $L_U$ be an $\F_q$-linear set of rank $(r-1)n$ in $\PG(r-1,q^n)$ and assume that $w\geq 2$ is maximal such that the points of $L_U$ have weight at least $w$. Then $w \mid n$ and $U$ is $\F_{q^w}$-linear.
	\end{proposition}
	\begin{proof}
		For $r=2$ this is Result \ref{CsMP} so we may assume $r>2$.
		Since the group $\PGL(r,q^n)$ is $2$-transitive we can always assume that the point $E_{r-1}\notin L_U$ and that a point of minimum weight $w$ lies on the line
		\[m:X_0=X_1=\ldots=X_{r-2}.\] Hence
		\[U=\{(x_0,x_1,\ldots,x_{r-2},F)\colon x_i\in\F_{q^n}, i=0,1,\ldots,r-2\},\]
		where $F=F(x_0,x_1,\ldots,x_{r-2}):\F_{q^n}^{r-1}\rightarrow \F_{q^n}$ is $\F_q$-linear, i.e. \[F(x_0,\ldots,x_{r-2})=\sum_{j=0}^{r-2}\sum_{i=0}^{n-1}\alpha_{j,i}x_j^{q^i}.\]
		
		Consider the line
		\[r_0\colon X_1=X_2=\ldots=X_{r-2}=0.\]
		
		Then \[r_0\cap L_U=\{\la(x_0,0,\ldots,0,F_0)\ra_{\F_{q^n}} :x_0\in\F_{q^n}^*\},\] where $F_0=F(x_0,0,\ldots,0)=\sum_{i=0}^{n-1}\alpha_{0,i}x_0^{q^i}$. Then $r_0\cap L_U$ is an $\F_q$-linear set of rank $n$ on a $\PG(1,q^n)$ and hence $F_0$ is $\F_{q^{w_0}}$-linear, where $w\leq w_0 \mid n$ and $w_0$ is the minimum weight of the points of $r_0$ in $L_U$, cf. Result \ref{CsMP}.
		
		Using the same arguments as above for each line
		\[r_i \colon X_j=0, \mbox{$\forall j\in\{0,1,\ldots,r-2\}$, $j\ne i$},\] with $i=1,\ldots,r-2$, we have that
		$F(x_0,x_1,\ldots,x_{r-2})=F_0(x_0)+\ldots+F_{r-2}(x_{r-2})$ and $F_i(x_i)$ is $\F_{q^{w_i}}$-linear with $1<w_i \mid n$.
		
		Also
		\[m\cap L_U=\{\langle(x,\ldots,x,F(x))\rangle_{\F_{q^n}}\colon x \in \F_{q^n}^*\},\] where
		\[F(x)=F_0(x)+\ldots+F_{r-2}(x)\] defines an $\F_q$-linear set of minimum weight $w$ and by Result \ref{CsMP} $F(x)$ is  $\F_{q^w}$-linear. Since $F(x_0,x_1,\ldots,x_{r-2})=F_0(x_0)+\ldots+F_{r-2}(x_{r-2})$ we have $1<w \mid \gcd(w_0,w_1,\ldots,w_{r-2})$. Thus the maximum field of linearity of $F$ is $\F_{q^w}$.
	\end{proof}

	\begin{lemma}\label{tangenti}
		If $L_U$ is a linear set of $\PG(r-1,q^n)$, $r>2$, of rank $m \leq n(r-2)$, then for each $P\in L_U$ there exist $r-1$ tangent lines to $L_U$ incident with $P$ and spanning $\PG(r-1,q^n)$.
	\end{lemma}
	\begin{proof}
		The number of lines incident with $P\in L_U$ is $(q^{n(r-1)}-1)/(q^n-1)$, among these lines, there are at most $(|L_U|-1)/q$ non-tangent lines.
		The size of $L_U$ is at most $(q^m-1)/(q-1)$. It follows that there are at least
		\[(q^{n(r-1)}-1)/(q^n-1)-((q^m-1)/(q-1)-1)/q >
		(q^{n(r-2)}-1)/(q^n-1)\]
		tangent lines incident with $P$ and hence they cannot be contained in the same hyperplane.
	\end{proof}

	\begin{theorem}
		For some integer $a$ let $L_U$ be an $\F_q$-linear set of rank $an$ in $\PG(r-1,q^n)$ such that the points of $L_U$ have weight at least $w \geq 2$ and $w$ is maximal with this property. Then $w \mid n$ and $U$ is $\F_{q^w}$-linear.
	\end{theorem}
	\begin{proof}
		We have already seen that the result holds when $a=r-1$, thus we may assume $2<r$ and $0<a<r-1$. Assume that the result holds in $\PG(r-2,q^n)$.
		Let $L_U$ be an $\F_q$-linear set of rank $an$ in $\PG(r-1,q^n)$ with minimum weight $w$, and let $P$ be a point of weight $w$.
		
		If $L_U$ meets each $(r-a-1)$-subspace of $\PG(r-1,q^n)$, then it is called an $a$-blocking set and an $a$-fold blocking set has size at least $q^{na}+q^{n(a-1)}+\ldots+q^n+1$ (see \cite{BB66}).
		The rank of $L_U$ is $an$ and hence its size cannot reach this bound. It follows that there is an $(r-a-1)$-subspace $\pi$ of $\PG(r-1,q^n)$ disjoint from $L_U$. Let $P$ be a point of minimum weight and define $\Sigma=\la P, \pi\ra=\PG(W,\F_{q^n})\cong \PG(r-a,q^n)$. Since $\pi$ is disjoint from $L_U$, the rank of $L_U \cap \Sigma=L_{U\cap W}$ is at most $n$. On the other hand, by Grassmann's identity, the rank of $L_{U\cap W}$ is at least $n$ and hence it is exactly $n$. Let $t_1,t_2,\ldots,t_{r-a}$ denote tangents to $L_U$ at $P$ spanning $\Sigma$, cf. Lemma \ref{tangenti}.
		
		Up to the action of the group $\PGL(r,q^n)$ we can assume that $E_0=P$ and $E_{r-i}=\pi\cap t_i$, with $i\in\{1,\ldots,r-a\}$.
		Then
		\[U=\{(x_0,x_1,\ldots,x_{a-1},F_a,F_{a+1},\ldots,F_{r-1}), x_i \in \F_{q^n},i=0,1,\ldots, a-1\},\] where $F_j=F_j(x_0,x_1,\ldots,x_{a-1}):\F_{q^n}^a\rightarrow \F_{q^n}$ is an $\F_q$-polynomial, i.e.
		\[F_j(x_0,\ldots,x_{a-1})=\sum_{k=0}^{a-1}\sum_{i=0}^{n-1}\alpha_{j,k,i}x_k^{q^i},\] for $j=a,a+1,\ldots, r-1$.
		
		By projecting $L_U$ from $E_{r-i}$, $i\in \{1,2,\ldots,r-a\}$ over the hyperplane $X_{r-i}=0$ we get the linear set $L_{U_{r-i}}$ of rank $an$ with
		\[U_{r-i}=\{(x_0,\ldots,x_{a-1},F_a,\ldots,F_{r-i-1},0,F_{r-i+1},\ldots,F_{r-1})\colon x_i \in \F_{q^n},\, 0\leq i \leq a-1\},\] and with minimum weight $w$ (represented by the point $E_0$). By the induction hypothesis, $F_{r-j}$ is $\F_{q^w}$-linear for $j\neq i$. Since this holds for each $i\in \{1,2,\ldots,r-a\}$, the result follows.
	\end{proof}
	
	Finally, we formulate Theorem \ref{thm:main1}  and its dual in terms of vector subspaces.

	\begin{theorem}
		\label{d1}
		Let $V$ denote an $r$-dimensional $\F_{q^n}$-subspace and let $U$ denote an $\F_q$-subspace of $V$ of dimension $m$, $n$ divides $m$, such that there is no $1$-dimensional $\F_{q^n}$-subspace meeting $U$ in an $\F_q$-subspace of dimension $1$. Then $U$ is $\F_{q^d}$-linear with some $d \mid n$ where $d=\min \{ \dim_{\F_q} (U \cap \la {\bf v}\ra_{\F_{q^n}}) : {\bf v}\in V\setminus\{\mathbf{0}\} \}$.
	\end{theorem}
	
	Let $\sigma: \Lambda\times \Lambda\longrightarrow \F_{q^n}$ be a non-degenerate
	reflexive sesquilinear form on $\Lambda=V(r,q^n)$ and define
	\begin{equation}\label{form:traccia}
		\sigma' \colon ({\bf u}, {\bf v})\in \Lambda\times \Lambda \rightarrow
		\mathrm{Tr}_{q^n/q}(\sigma({\bf u}, {\bf v}))\in \F_q,
	\end{equation}
	where $\mathrm{Tr}_{q^n/q}$ denotes the trace function of $\F_{q^n}$ over $\F_q$.
	Then $\sigma'$ is a non-degenerate reflexive sesquilinear form on
	$\Lambda$, when $\Lambda$ is  regarded as an $rn$-dimensional vector space  over
	$\F_{q}$. Let $\tau$ and $\tau'$ be the orthogonal complement maps
	defined by $\sigma$ and $\sigma'$ on the lattices  of the
	$\F_{q^n}$-subspaces and $\F_{q}$-subspaces of $\Lambda$, respectively.
	Recall that  if $R$ is an
	$\F_{q^n}$-subspace of $\Lambda$ and $U$ is an $\F_{q}$-subspace of $\Lambda$
	then $U^{\tau'}$ is an $\F_q$-subspace of $\Lambda$, $\dim_{q^n}R^\tau+\dim_{q^n}R= r$ and
	$\dim_q U^{\tau'}+\dim_q U= rn$. It easy to see
	that $R^\tau=R^{\tau'}$ for each $\F_{q^n}$-subspace $R$ of $\Lambda$.
	
	Also, $U^{\tau '}$ is called the \emph{dual} of $U$ (w.r.t. $\tau'$).
	Up to $\Gamma\mathrm{L}(r,q^n)$-equivalence, the dual of an $\F_q$-subspace of $\Lambda$ does not depend on the choice of the non-degenerate reflexive sesquilinear forms $\sigma$ and $\sigma'$ on $\Lambda$. For more details see \cite{OP2010}. If $R$ is an $s$-dimensional $\F_{q^n}$-subspace of $\Lambda$ and $U$ is a $t$-dimensional $\F_q$-subspace of $\Lambda$, then
	\begin{equation}\label{pesi}
		\dim_{\F_q}(U^{\tau'}\cap R^\tau)-\dim{\F_q}(U\cap R)=rn-t-sn.
	\end{equation}
	
	\begin{theorem}
		\label{d2}
		Let $V$ denote an $r$-dimensional $\F_{q^n}$-subspace and let $U$ denote an $\F_q$-subspace of $V$ of dimension $m$, $n$ divides $m$, such that there is no hyperplane of $V$ meeting $U$ in an $\F_q$-subspace of dimension $m-n+1$. Then $U$ is $\F_{q^d}$-linear with some $d\mid n$ where $m-n+d=\min \{ \dim_{\F_q} (U \cap \la {\bf v}\ra_{\F_{q^n}}^{\tau}) : {\bf v}\in V\setminus\{\mathbf{0}\} \}$.
	\end{theorem}
	\begin{proof}
		By Equation \eqref{pesi} the $1$-dimensional $\F_{q^n}$-subspaces meet $U^{\tau'}$ in $\F_q$-subspaces of dimension $0$ or of dimension at least $2$. Then $U^{\tau'}$ is $\F_{q^d}$-linear for some $d \mid n$, where $m-n+d=\min \{ \dim_{\F_q} (U \cap \la {\bf v}\ra_{\F_{q^n}}^{\tau}) : {\bf v}\in V\setminus\{\mathbf{0}\} \}$.  Then also $U$ is $\F_{q^d}$-linear.
	\end{proof}
	
	In \cite[Section 4]{Olga2022}, $e$-divisible rank metric codes are studied and similar conclusions are deduced in the special case when each hyperplane meets $U$ in an $\F_q$-subspace of dimension a multiple of $e$, $1 < e \mid n \mid m$.

	\section{Proof of Theorem \ref{thm:main}}
	
	Let us start by describing the geometric setting we adopt to study $\FF_q$--linear sets of $\PG(V,\F_{q^n})= \PG(r-1,q^n)$ (see \cite{Lu1999} and \cite{giuzzipepe15}).
	
	Regarding the $r$-dimensional $\FF_{q^n}$-vector space $V$ as an $\FF_q$--vector space of dimension $rn$, the points of $\PG(r-1,q^n)$ corresponds to a partition of $\PG(rn-1,q)$ into $(n-1)$-dimensional projective subspaces. Such a partition, say $\mathcal S$, is called \textit{Desarguesian spread} of $\PG(rn-1,q)$ and the pair $(\PG(rn-1,q),{\mathcal S})$ is said to be the $\FF_q$-\textit{linear representation} of $\PG(r-1,q^n)$. In this setting, an $\FF_q$-linear set $L_U$ of $\PG(r-1,q^n)$ is the subset consisting of the elements of $\mathcal{S}$ with non-empty intersection with the projective subspace $\PG(U,\FF_q)$ of $\PG(rn-1,q)$ defined by $U$.
	
	Consider the following cyclic representation of $\PG(rn-1,q)$ in $\PG(rn-1,q^n)$. Let $\PG(rn-1,q^n)= \PG(V',\F_{q^n)}$, where $V'=\F_{q^n}^{rn}$ and let ${\mathbf e}_i$ be the $i$-th canonical basis element of $V'$. Let $\sigma$ be the $\F_q$-semilinear map defined by the rule ${\mathbf e}_i\mapsto {\mathbf e}_{i+r}$, where the subscripts are taken mod $rn$ and with accompanying automorphism $x\in\F_{q^n}\mapsto x^q\in\F_{q^n}$. Then $\sigma$ has order $n$ and the $\FF_q$-vector subspace $\mathrm{Fix}\,\sigma=\{(\mathbf{x},\mathbf{x}^q,\ldots,\mathbf{x}^{q^{n-1}}),\mathbf{x}=x_0,x_1,\ldots,x_{r-1} \mbox{ and }x_i \in \FF_{q^n}\}$ of $V'$ defines a set of points of $\PG(rn-1,q^n)$  fixed by $\sigma$. i.e. a subgeometry $\Sigma=\PG(rn-1,q)$. The elements of $\mathcal{S}$ are the subspaces $\Pi_P:=\langle P,P^{\sigma},\ldots,P^{\sigma^{n-1}} \rangle \cap \Sigma$, with $P\in \Pi_0 \cong \PG(r-1,q^n)$ and $\Pi_0=\PG(V_0,\FF_{q^n})$, where $V_0:=\{({\mathbf x}, {\mathbf 0},\ldots,{\mathbf 0})\colon {\mathbf x}=x_0,\ldots,x_{r-1} \mbox{ and }x_i \in \FF_{q^n}\}$ (see \cite{Lu1999}). Let $\Pi_i=\PG(V_i,\FF_{q^n})$ be $\Pi_0^{\sigma^i}$, i.e. $V_i=\langle \mathbf{e}_{h+ir}, h=0,1,\ldots,r-1\rangle$. In the following, we shall identify a point $P$ of $\Pi_0= \PG(r-1,q^n)$ with the spread element $\Pi_P$. We observe that $P$ is just the projection of $\Pi_P$ from $\langle \Pi_1,\Pi_{2},\ldots,\Pi_{n-1}\rangle$ on $\Pi_0$. If $L_U$ is a linear set of rank $m$, then it is induced by an $(m-1)$-dimensional projective subspace $\PG(U,\FF_q) \subset \Sigma$ and it can be viewed both as the subset of $\Pi_0$ that is the projection of $\PG(U,\F_q)$ from $\langle \Pi_1,\Pi_{2},\ldots,\Pi_{n-1}\rangle$ on $\Pi_0$ as well as the subset of $\mathcal{S}$ consisting of the elements $\Pi_P$ such that $\Pi_P \cap \PG(U,\FF_q) \neq \emptyset$.  We stress out that we have defined the subspaces $\PG(U,\FF_q)$ and $\Pi_P$ as subspaces of $\Sigma=\PG(rn-1,q)$. In particular we will denote by $\PG(U,\FF_{q^n})$ the unique projective subspace of $\PG(rn-1,q^n)$ with the same dimension as $\PG(U,\F_q)$ such that $\PG(U,\FF_{q^n})\cap\Sigma=\PG(U,\FF_q)$, i.e. $\PG(U,\F_{q^n})$ is induced by the $\F_{q^n}$-vector subspace $\la U \ra_{\F_{q^n}}$ of $V'$.
	
	Recall that we are assuming $r-1\leq \dim\PG(U,\F_q)\leq(r-1)n-1$. Note that, if a point $P$ has weight $i$ in $L_U$ then \[\dim\left( \Pi_P\cap \PG(U,\FF_q)\right)=\dim\left(\langle P,P^\sigma,\ldots,P^{\sigma^{n-1}}\rangle\cap\PG(U,\FF_{q^n})\right)=i-1,\] and since $\Pi_P$ is projected from $\langle \Pi_1,\ldots,\Pi_{n-1}\rangle$ to the point $P$ of $\Pi_0$, we get
	\[\dim\left(\langle P,P^\sigma,\ldots,P^{\sigma^{n-1}}\rangle\cap\PG(U,\FF_{q^n})\cap \la \Pi_1,\Pi_2,\cdots,\Pi_{n-1} \ra\right)=i-2.\]

	\bigskip
	
	The image under the Grassmann embedding $\varepsilon_n$ of a Desarguesian spread $\cS$ of $\PG(rn-1,q)$ determines the algebraic variety $\cV_{rn} \subset \PG(r^n-1,q)$.
	In more details, put $P=\langle (x_0,x_1,\ldots,x_{r-1},0,0,\ldots,0)\rangle$, then the Grassmann embedding of $\Pi_P$ is the point defined by the vector of the minors of order $n$ of

	\[\left(
	\begin{array}{ccccccccccccc}
		x_0 & x_1 &\cdots & x_{r-1} & 0 & 0 & \cdots & \cdots & \cdots &0 & 0 &\cdots & 0  \\
		0 & 0 &\cdots  & 0 & x_0^q & x_1^q & \cdots & x_{r-1}^q& \cdots & \cdots &\cdots & \cdots  & \cdots\\
		\cdots & \cdots & \cdots & \cdots & \cdots & \cdots & \cdots & \cdots & \cdots & \cdots& \cdots & \cdots & \cdots\\
		0 & 0 & \cdots & 0 & 0 & 0 & \cdots & 0 & \cdots &  x_0^{q^{n-1}}& x_1^{q^{n-1}} &\cdots &  x_{r-1}^{q^{n-1}} \\
	\end{array}
	\right).\]
	
	If we disregard the identically zero minors, then the Grassmann embedding of the Desarguesian spread $\mathcal{S}$, i.e. $\cV_{rn}$,   is the image of the map
	\[ \alpha:\langle (x_0,\ldots,x_{r-1})\rangle\in\PG(r-1,q^n)\mapsto
	\left\langle \left(\prod_{i=0}^{n-1} x_{f(i)}^{q^i}\right)_{f \in \fF}\right\rangle\in \PG(r^n-1,q)
	\subset \PG(r^n-1,q^n),\]
	where $\fF=\{ f: \{0,\ldots, n-1\}\to\{0,\ldots,r-1\} \}$.
	Here, $\alpha$ is the map that makes the following diagram commute:
	$$\begin{tikzcd}[row sep=large]
		\PG(r-1,q^n)\arrow[r, dotted, "\alpha"] \arrow[d ,"\text{$\FF_q$-linear repres.}"'] &\PG(r^n-1,q)  \\
		\mbox{\begin{minipage}{6cm}\begin{centering}
					$\PG(rn-1,q)$ \\ $\cS =$ {Desarguesian Spread} \end{centering}
		\end{minipage}}
		\arrow[ru, ->, "\varepsilon_n"'] & \\
	\end{tikzcd}
	.$$
	For more details see \cite{giuzzipepe15}.
	
	Let $\mathbf{x}^{(i)}:=(x_0^{(i)},x_1^{(i)},\ldots,x_{r-1}^{(i)})\in \mathbb{F}^r$, where $\mathbb{F}$ is any field. Then the Segre variety
	
	\[\displaystyle \Sigma_{rn}:=\underbrace{\PG(r-1,\mathbb{F})\otimes \PG(r-1,\mathbb{F})\otimes \cdots \otimes \PG(r-1,\mathbb{F})}_{n \text{ times}}\subset \PG(r^n-1,\mathbb{F})\]
	
	is the image of the map 

$$s: \PG(r-1,\mathbb{F})\times \PG(r-1,\mathbb{F})\times\cdots \times \PG(r-1,\mathbb{F}) \rightarrow \PG(r^n-1,\mathbb{F})$$ such that:
	
	\[s:\left(\langle\mathbf{x}^{(0)}\rangle,\langle\mathbf{x}^{(1)}\rangle,\ldots,\langle\mathbf{x}^{(n-1)}\rangle\right) \mapsto \left\langle \left(\prod_{i=0}^{n-1} x_{f(i)}^{(i)}\right)_{f \in \fF}\right\rangle.\]
	
	Let $\mathcal{T}$ be the collection of subspaces of $\PG(rn-1,q^n)$ of type $\la P_0, P_1,\ldots,P_{n-1}\ra$ with $P_i=\la v_i \ra \in \Pi_i$, then $\Sigma_{rn}$ is the Grassmann embedding of $\mathcal{T}$, that is, for $v_i \in V_i$, by abuse of notation, we write $v_0\wedge v_1\wedge \cdots \wedge v_{n-1}=v_0\otimes v_1 \otimes \cdots \otimes v_{n-1}$.
	
	Hence we observe that $\cV_{rn}$ is the subvariety  of $\Sigma_{rn}$ fixed by the semilinear map \[\hat\sigma:\ v_0\otimes v_1 \otimes \cdots \otimes v_{n-1}\mapsto v_{n-1}^{\sigma}\otimes v_0^{\sigma} \otimes \cdots \otimes v_{n-2}^{\sigma},\] with $\la v_i \ra_{\F_{q^n}} \in \Pi_i \cong \PG(r-1,q^n)$.

	We recall the following result, which we will explain below.
	
	\begin{theorem}\cite{giuzzipepe15}
		The Grassmann embedding of a linear set $L_U$ of rank $m$ of $\PG(r-1,q^n)$ is the intersection of $\cV_{rn}$
		with a linear subspace.
		In particular, if the rank of $L_U$ is maximum, then the image of the linear set is a hyperplane section of $\cV_{rn}$.
	\end{theorem}

	Let $B$ be the $m\times rn$  matrix whose rows are an $\FF_q$-basis of $U$. Since $\dim_{q^n}\langle U\rangle_{\F_{q^n}}=\dim_q\, U$, the rows of $B$ are an $\F_{q^n}$-basis of $\langle U\rangle_{\F_{q^n}}$ as well.
	If $P_i=\la v_i \ra \in \Pi_i$, with $v_i=(0,0,\ldots,\mathbf{x}^{(i)},0,\ldots,0)$, then $\PG(U,\F_{q_n})\cap \langle v_0,v_1,\ldots,v_{n-1}\rangle \ne \emptyset$ if and only if the rank of the matrix
	
	\[\left(
	\begin{array}{ccccccccccccc}
		&  &  &  &  &  & B &  &  &  &  &  &  \\
		x_0^{(0)} & x_1^{(0)} &\cdots & x_{r-1}^{(0)} & 0 & 0 & \cdots & \cdots & \cdots &0 & 0 &\cdots & 0  \\
		0 & 0 &\cdots  & 0 & x_0^{(1)} & x_1^{(1)} & \cdots & x_{r-1}^{(1)}& \cdots & 0 & 0 & \cdots  & 0 \\
		\cdots & \cdots & \cdots & \cdots & \cdots & \cdots & \cdots & \cdots & \cdots & \cdots& \cdots & \cdots & \cdots\\
		0 & 0 & \cdots & 0 & 0 & 0 & \cdots & 0 & \cdots &  x_0^{(n-1)}& x_1^{(n-1)} &\cdots &  x_{r-1}^{(n-1)} \\
	\end{array}
	\right)\]
	is less than $n+m$.
	
	Hence the Grassmann embedding of the elements of $\mathcal T$ with non-empty intersection with $\PG(U,\F_{q^n})$ is an algebraic variety $X(\F_{q^n})=V(f_1,f_2,\ldots,f_k)$, where $f_h=f_h(\mathbf{x}^{(0)},\mathbf{x}^{(1)},\ldots,\mathbf{x}^{(n-1)})$ are linear polynomials in each set $\{\mathbf{x}^{(i)}\}$ of variables. Also, $X(\F_{q^n})$ turns out to be a linear section of the Segre variety $\Sigma_{rn}$. We stress out that $f_h(v_0,v_1,\ldots,v_{n-1})=0$ $\forall \, h$ implies that $\la v_0,v_1,\ldots,v_{n-1}\ra\cap \PG(U,\F_{q^n})\neq \emptyset$ only if $v_i\in V_i \setminus \{\mathbf{0}\}$.

	It follows that the Grassmann embedding of the linear set $L_U$ is the algebraic variety
	\[X=X(\F_{q^n})\cap\mathrm{Fix}\,\hat\sigma,\]
	which is also the linear section of $\mathcal{V}_{rn}$ defined by the condition that the rank of

	\[\left(
	\begin{array}{ccccccccccccc}
		&  &  &  &  &  & B &  &  &  &  &  &  \\
		x_0 & x_1 &\cdots & x_{r-1} & 0 & 0 & \cdots & \cdots & \cdots &0 & 0 &\cdots & 0  \\
		0 & 0 &\cdots  & 0 & x_0^q & x_1^q & \cdots & x_{r-1}^q& \cdots & 0 & 0 & \cdots  & 0 \\
		\cdots & \cdots & \cdots & \cdots & \cdots & \cdots & \cdots & \cdots & \cdots & \cdots& \cdots & \cdots & \cdots\\
		0 & 0 & \cdots & 0 & 0 & 0 & \cdots & 0 & \cdots &  x_0^{q^{n-1}}& x_1^{q^{n-1}} &\cdots &  x_{r-1}^{q^{n-1}} \\
	\end{array}
	\right)\]
	is less than $m+n$.
	
	\bigskip

	The aim of the following theorems is to describe the points of weight at least $2$ as singular points of $X(\F_{q^n})$.

	\begin{theorem}\label{singular}
		Let $v_i\in V_i\setminus\{\mathbf{0}\}$, $i\in\{0,\ldots,n-1\}$. If  $\dim \la v_0, v_1,\ldots,v_{n-1} \ra\cap \PG(U,\FF_{q^n})\geq 1$, then $f_h$ and every partial derivative of $f_h$ vanishes in $v_0+v_1+\ldots + v_{n-1}$ $\forall h=1,2,\ldots,k$. On the other hand if $f_h$ and every partial derivative of $f_h$ vanishes in $v_0+v_1+\ldots + v_{n-1}$ $\forall h=1,2,\ldots,k$, then there exists an integer $j\in\{0,\ldots,n-1\}$ such that $\la v_i\colon i=0,1,\ldots,n-1, i\neq j \ra \cap \PG(U,\FF_{q^n}) \neq \emptyset$.
	\end{theorem}
	\begin{proof}
		We observe that $\la v_0,v_1,\ldots,v_{n-1} \ra\cap \PG(U,\F_{q^n}) \neq \emptyset$ if and only if $f_h$ vanishes in $v_0+v_1+\ldots + v_{n-1}$ $\forall h=1,2,\ldots,k$.
		
		If $\dim \la v_0,v_1,\ldots,v_{n-1} \ra\cap \PG(U,\F_{q^n}) \geq 1$, then any hyperplane of $\la v_0,v_1,\ldots,v_{n-1} \ra$ has non-empty intersection with $\PG(U,\FF_{q^n})$, so $\la v_i : i \neq j \ra \cap \PG(U,\FF_{q^n}) \neq \emptyset$ for any fixed $j \in \{0,1,\ldots,n-1\}$. Therefore, for every fixed $j \in \{0,1,\ldots,n-1\}$, $\la  w_0,  w_1,\ldots, w_{n-1} \ra \cap \PG(U,\FF_{q^n})\neq \emptyset$ for every choice of $w_j\in V_j\setminus\{\mathbf{0}\}$ and $w_t=v_t, t \neq j$. By the linearity of $f_h$ with respect to each set of variables $\{x_0^{(i)},x_1^{(i)},\ldots,x_{n-1}^{(i)}\}$, we get that any of the $rn$ partial derivative of $f_h$ vanishes in $v_0+v_1+\cdots + v_{n-1}$.
		
		Now suppose that $f_h$ and every partial derivative of $f_h$ vanishes in $v_0+v_1+\ldots + v_{n-1}$. By $f_h(v_0,v_1,\ldots,v_{n-1})=0$ $\forall \, h$, we get $\la v_0,v_1,\ldots,v_{n-1}\ra \cap \PG(U,\F_{q^n}) \neq \emptyset$. Again, since $f_h$ is linear in each set of variables $\{x_0^{(i)},x_1^{(i)},\ldots,x_{n-1}^{(i)}\}$ and every partial derivative of $f_h$ vanishes in $v_0+v_1+\cdots + v_{n-1}$, we get that $\la u,v_i\colon i=0,1,\ldots,n-1, i \neq j \ra\cap \PG(U,\F_{q^n}) \neq \emptyset$ $\forall \, u \in V_j\setminus \{\mathbf{0}\}$, $\forall \, j \in \{0,1,\ldots,n-1\}$. Suppose by contradiction that $\la v_i\colon i=0,1,\ldots,n-1, i \neq j \ra \cap \PG(U,\FF_{q^n}) = \emptyset$ $\forall \, j=0,1,\ldots, n-1$. Then, note that
		\begin{equation}\label{form1}
			\PG(U,\FF_{q^n})\cap\langle v_0,v_1,\ldots,v_{n-1}\rangle=\la \lambda_0v_0+\lambda_1v_1+\ldots+\lambda_{n-1}v_{n-1}\ra,
		\end{equation} for some $\lambda_i\in \F_{q^n}\setminus \{0\}$,
		and in order to have $\la u,v_i\colon i=0,1,\ldots,n-1, i \neq j\ra\cap \PG(U,\F_{q^n}) \neq \emptyset$ $\forall \, u \in V_j\setminus \{\mathbf{0}\}$, we must have $\dim \la v_i, V_j\colon i=0,1,\ldots,n-1, i\neq j \ra \cap \PG(U,\F_{q^n}) \geq r-1$  $\forall \, j=0,1,\ldots, n-1$. Let $$T_j:=\la v_i, V_j\colon i=0,1,\ldots,n-1, i \neq j \ra \cap \PG(U,\F_{q^n}).$$
		Since $\la v_i, V_j\colon i=0,1,\ldots,n-1, i \neq j \ra \cap \displaystyle\bigoplus_{k\neq j}\la v_i, V_k\colon i=0,1,\ldots,n-1, i \neq k \ra=\la v_0,v_1,\ldots,v_{n-1} \ra$,
		we have $T_j \cap \la T_k, k=0,1,\ldots,n-1, k\neq j \ra \subseteq\la v_0,v_1,\ldots,v_{n-1} \ra $ and by \eqref{form1} we get $T_j \cap \la T_k, k=0,1,\ldots,n-1, k\neq j \ra\subseteq \la\lambda_0v_0+\lambda_1v_1+\ldots+\lambda_{n-1}v_{n-1}\ra$. Hence, in the quotient space with respect to such a point each $T_j$ has projective dimension at least $r-2$ and the $\F_{q^n}$-subspaces defining $T_0,T_1,\ldots,T_{n-1}$ are in direct sum. Thus $\PG(U,\F_{q^n})$ contains a projective subspace of dimension at least $(r-1)n$, contradicting the assumption $\dim \PG(U,\F_{q^n}) \leq (r-1)n-1$.
	\end{proof}

	From the previous theorem, the following result holds true.
	\begin{corollary}
		\label{cor:singular}
		Let $L_U$ be an $\FF_q$-linear set of $\Pi_0\simeq \PG(r-1,q^n)$ and let $P:=\la v \ra \in \Pi_0$. Then if $P$ is a point of $L_U$ of weight at least $2$ then every partial derivative of $f_h$ vanishes in $v+v^\sigma+\ldots + v^{\sigma^{n-1}}$ for each $h$.
	\end{corollary}
	\begin{proof}
		Since $\la v,v^\sigma,\ldots, v^{\sigma^{n-1}}\ra_{\F_q} \cap \PG(U,\F_q)$ is set-wise fixed by $\sigma$, by  \cite[Lemma1]{Lu1999}, $\dim \la v,v^\sigma,\ldots, v^{\sigma^{n-1}}\ra_{\F_q} \cap \PG(U,q)=\dim \la v,v^\sigma,\ldots, v^{\sigma^{n-1}}\ra_{\F_{q^n}} \cap \PG(U,q^n)$. By Theorem \ref{singular} we get the statement.
	\end{proof}

	The following result is folklore, we include it for sake of completeness.
	
	\begin{lemma}
		\label{lem}
		Let $U$ be an $\F_q$-subspace $U$ of $\F_{q^n}^t$ of dimension $m$ such that $\dim \langle U \rangle_{\F_{q^n}}=m$. If $f\in \F_{q^n}[x_0,x_1,\ldots,x_{t-1}]$ is homogeneous of degree $d\leq q$ that vanishes on $L_U$
		then $f$ vanishes also in the vectors of $\la U \ra_{\F_{q^n}}$.
	\end{lemma}
	\begin{proof}
		Up to the action of the group $\mathrm{GL}(t,q^n)$ we can suppose that $\la U\ra_{\F_{q^n}}$ is defined by the equations $x_i=0,\, \forall i=m,m+1,\ldots,t-1$. Also, since the stabilizer of $\la U\ra_{\F_{q^n}}$ in $\mathrm{GL}(t,q^n)$ acts transitively on the $\F_q$-subspaces having dimension $m$ over $\F_q$, we can assume that the vectors of $U$ are elements of $\la U\ra_{\F_{q^n}}$ with $x_i\in \F_q\, \forall i=0,1,\ldots,m-1$. Then the ideal of polynomials vanishing on the projective set defined by $U$ is generated by $x_ix_j^q-x_i^qx_j,\,0\leq i < j \leq m-1$, $x_m,\ldots, x_{t-1}$.  Since the degree of  $f$ is at most $q$, $f$ is in the ideal generated by $x_m,\ldots,x_{t-1}$ and hence $f$ vanishes on $\la U \ra_{\F_{q^n}}$.
		%, a polynomial of degree $d\leq q$ cannot be in the ideal of an $\F_q$-subspace and it can vanish on the subspace only if it is 0.
	\end{proof}

	Then we get the following.
	
	\begin{corollary}\label{reducibility}
		Let $L_U$ be an $\F_q$--linear set of $\PG(r-1,q^n)$ of rank $m$, with $n\leq q$, such that every point has weight at least $2$ and let $U_i:=\la U \ra _{\F_{q^n}}\cap \displaystyle \opl_{j \neq i}V_j$. Then $\la v_0,v_1,\ldots,v_{n-1}\ra \cap \la U \ra_{\F_{q^n}} \neq \{\mathbf{0}\}$ if and only if $\la v_0,v_1,\ldots,v_{n-1}\ra \cap U_i \neq \{\mathbf{0}\}$ for some $i \in \{0,1,\ldots,n-1\}$.
	\end{corollary}
	\begin{proof}
		Since every point of $L_U$ has weight at least $2$,  by Corollary \ref{cor:singular}, every partial derivative of $f_h$   vanishes on each vector of $U$, $\forall \, h$. By hypothesis,  $n\leq q$ and $\deg f_h=n$, hence we can apply Lemma \ref{lem} and so $f_h$ and every partial derivative of $f_h$  vanishes on each vector of $\la U \ra_{\F_{q^n}}$, $\forall \, h=1,2,\ldots,k$. By Theorem \ref{singular} we get the statement.
	\end{proof}
	
	We observe that by Corollary \ref{reducibility} the variety $X(\F_{q^n})$ is reducible, that is  $X(\F_{q^n})=Y_0 \cup Y_1 \cup \cdots Y_{n-1}$, where $Y_i$ is the Grassmann embedding of the elements of $\mathcal{T}$ with non zero intersection with $U_i$.
	
	We clearly have $U_i^{\sigma}=U_{i+1}$ (with the subscripts modulo $n$). Let $S$ be the largest subgroup of $\la \sigma\ra$ fixing $U_0$. Clearly we cannot have $S=\la \sigma\ra$, hence $S=\la \sigma^d \ra$ for some $d>1$ divisor of $n$. It follows that $U_i^{\sigma^d}=U_i$ $\forall \, i$.
	In particular, $U_i=\la U \ra_{\F_{q^n}}\cap \displaystyle\opl_{k \not\equiv i \pmod d} V_k$.
	
	Let $\overline{U}:=\overline{U}_0\oplus \overline{U}_1\oplus \cdots \oplus \overline{U}_{d-1}$, where $\overline{U}_i:=\la \la U\ra_{\F_{q^n}}, V_h \colon  h \not\equiv i \pmod d \ra\cap \displaystyle\opl_{k\equiv i \pmod d} V_k$.

	\begin{lemma}\label{inclusion}
		We have $\la U \ra_{\F_{q^n}}\leq \overline{U}$.
	\end{lemma}
	\begin{proof}
		Every vector of $v\in V=\displaystyle\opl_{i=0}^{n-1}V_i$ can be written as $v=v_0+v_1+\cdots + v_{n-1}$ with $v_i$ being the projection of $v$ on $V_i$ from $\displaystyle\opl_{j\neq i}V_j$.
	\end{proof}

	\begin{theorem}
		The projection of the subspace $\overline{U} \cap Fix(\sigma)$ from $\displaystyle\opl_{j\geq 1} V_j$ on the $r$-dimensional $\F_{q^n}$-vector space $V_0$ is the $\F_{q^d}$-span of the projection of $U$.
	\end{theorem}
	\begin{proof}
		The space $\overline{U} \cap Fix(\sigma)$  is spanned by the subspaces $\la v,v^{\sigma},\ldots,v^{\sigma^{d-1}}\ra$ with $v \in \left(\displaystyle \opl_{j\equiv 0 \pmod d} V_j\right)\cap Fix(\sigma^d)$ such that $\la v,v^{\sigma},\ldots,v^{\sigma^{d-1}}\ra \cap U \neq \{\mathbf{0}\}$. For every $\lambda \in \F_{q^d}$, $\lambda v+\lambda^q v^{\sigma}+\cdots + \lambda^{q^{d-1}}v^{\sigma^{d-1}}$ is still in $\overline{U} \cap Fix(\sigma)$, hence the projection of $\overline{U} \cap Fix(\sigma)$  from $\displaystyle\opl_{j\geq 1} V_j$ on $V_0$ is the $\F_{q^d}$-span of $U$.
	\end{proof}
	
	We need the following technical lemma.
	
	\begin{lemma}\label{span}
		Let $v=v_0+v_1+\cdots + v_{n-1} \in V$ with $v_i \in V_i$. If the dimension of $\la v^{\sigma^i} : i=0,1,\ldots,n-1\ra$ over $\F_{q^n}$ is $0 < k<n$, then $v_{k+i}$ is a non-trivial linear combination of  $\{v_0^{\sigma^{k+i}}, v_1^{\sigma^{k+i-1}},\ldots,v_{k-1}^{\sigma^{i+1}}\}$ $\forall \, i \in \{0,1,\ldots, n-k-1\}$. If there is a $v_j =\mathbf{0}$, then there exists $h \in \{0,1,\ldots,k-1\}$ such that $v_{k+i} \in \la v_l^{\sigma^{k+i-l}} : l=0,1,\ldots,k-1,\,l \neq h \ra$ $\forall \, i \in \{0,1,\ldots, n-k-1\}$ and $v_h^{\sigma^{k-h-1}} \in \la v_l^{\sigma^{k-l-1}} : l=0,1,\ldots,k-1,\,l \neq h  \ra$.
	\end{lemma}
	\begin{proof}
		$k>0$ and hence $v$ is not the null-vector. If $k=1$ then then the first part is trivial and it is easy to see that $v_j={\bf 0}$ would imply that $v$ is the null-vector, contradicting the fact that $k>0$. If $k>1$ then $\{v,v^{\sigma}\}$ is independent since otherwise $k=1$.
		
		To prove the first part first we prove by induction on $i$ that $\{v,v^{\sigma},\ldots,v^{\sigma^{i}}\}$ is independent for $1 \leq i \leq k-1$. For $i=1$ we have already seen this. Suppose that $k>2$ and the statement true for $i-1>1$.  If  $v^{\sigma^i} \in \la v,v^{\sigma},\ldots,v^{\sigma^{i-1}} \ra$, then  $v^{\sigma^{i+1}} \in \la v^{\sigma},v^{\sigma^2},\ldots,v^{\sigma^{i}} \ra=\la v,v^{\sigma},\ldots,v^{\sigma^{i-1}} \ra$  and on turn $v^{\sigma^j} \in \la v,v^{\sigma},\ldots,v^{\sigma^{i-1}} \ra$ $\forall \, j > i-1$, a contradiction since $i<k$.
		Hence $\{v,v^{\sigma},\ldots,v^{\sigma^{k-1}}\}$ is independent.
		Since $v^{\sigma^i}=v_{n-i}^{\sigma^i}+v_{n-i+1}^{\sigma^{i}}+\cdots + v_{n-i-1}^{\sigma^i}$,  by $v^{\sigma^i} \in  \la v,v^{\sigma},\ldots,v^{\sigma^{k-1}} \ra$  and projecting from $\displaystyle \opl_{h \neq k-1}V_h$ on $V_{k-1}$, we get that $v_{n-i+k-1}^{\sigma^i} \in \la v_{k-1}, v_{k-2}^{\sigma},\ldots, v_0^{\sigma^{k-1}} \ra$ $\forall \, i \geq k$.  This proves the first part of the statement. Suppose that  $v_j =\mathbf{0}$ for some $j$. If $j \leq k-1$, then take $h=j$ and we get the assertion. If $j \geq k$, then it means that $\{v_0^{\sigma^{k-1}}, v_1^{\sigma^{k-2}},\ldots,v_{k-1} \}$ is linearly dependent and hence $v_h^{\sigma^{k-h-1}} \in \la v_l^{\sigma^{k-l-1}} : l=0,1,\ldots,k-1,\,l \neq h  \ra$ for some $h \leq k-1$ and this completes the proof.
	\end{proof}

	\begin{theorem}\label{main}
		Let $r=2$ or $\la L_U \ra=\PG(r-1,q^n)$ for some $r>2$ and such that $m>(r-1)d$. If $n\leq q$ then $U$ and $\overline{U}\cap Fix(\sigma)$ define the same linear set.
	\end{theorem}
	\begin{proof}
		By Lemma \ref{inclusion}, the linear set defined by $U$ is contained in the one defined by $\overline{U}\cap Fix(\sigma)$. Viceversa, let $u+u^{\sigma}+\cdots + u^{\sigma^{n-1}} \in \overline{U}=\overline{U}_0\oplus \overline{U}_1\oplus \cdots \oplus \overline{U}_{d-1}$, with $u \in V_{0}\setminus\{\mathbf{0}\}$ then there exists $u'\in \la U \ra_{\F_{q^n}}\setminus \{\mathbf{0}\}$  such that the projection of $u'$ from $\displaystyle \opl_{h \not\equiv 0 \pmod d}V_h$ on $\displaystyle\opl_{k\equiv 0 \pmod d} V_k$ is $u+u^{\sigma^d}+\cdots + u^{\sigma^{n-d}}$, indeed, choose any non-zero $u' \in \la u+u^{\sigma^d}+\cdots + u^{\sigma^{n-d}}, V_h,  h \not\equiv 0 \pmod d \ra \cap \la U\ra_{\F_{q^n}}$.
		Since $\la u+u^{\sigma^d}+\cdots + u^{\sigma^{n-d}}, V_h,  h \not\equiv 0 \pmod d \ra \cap \la U\ra_{\F_{q^n}}$ is set-wise fixed by $\sigma^{d}$, $\dim_{\F_{q^d}} \la u+u^{\sigma^d}+\cdots + u^{\sigma^{n-d}}, V_h,  h \not\equiv 0 \pmod d \ra \cap \la U\ra_{\F_{q^n}} \cap Fix(\sigma^d)=\dim_{\F_{q^n}} \la u+u^{\sigma^d}+\cdots + u^{\sigma^{n-d}}, V_h,  h \not\equiv 0 \pmod d \ra \cap \la U\ra_{\F_{q^n}}$ (\cite[Lemma 1]{Lu1999}), hence we can take  $u' \in \la u+u^{\sigma^d}+\cdots + u^{\sigma^{n-d}}, V_h,  h \not\equiv 0 \pmod d \ra \cap \la U\ra_{\F_{q^n}}\cap Fix(\sigma^d)$. Therefore, there are $u_i \in V_{i},i=1,2,\ldots, d-1,$ such that $u'$ is a non-zero vector of
		$\la u, u_1,u_2,\ldots, u_{d-1}, u^{\sigma^d}, u_1^{\sigma^d},u_2^{\sigma^d},\ldots, u_{d-1}^{\sigma^d},\ldots, u^{\sigma^{n-d}}, u_1^{\sigma^{n-d}},u_2^{\sigma^{n-d}},\ldots, u_{d-1}^{\sigma^{n-d}}\ra \cap \la U \ra_{\F_{q^n}}$ (note that $u' \notin U_0$).
		\bigskip
		
		A vector of $\la u', U_0 \ra \setminus U_0$ is also projected on $u+u^{\sigma^d}+\cdots + u^{\sigma^{n-d}}$ (with the same vertex and axis as above). Since $L_{U}$ is non-degenerate, that is, $L_{U}$ spans $\PG(r-1,q^n)$, the projection of $\la U\ra_{\F_{q^n}}$ from $\displaystyle\opl_{i\geq 1} V_i$ on
		$V_0$ is $V_0$, hence the dimension of $U_0$ is $m-r$. The subspace  $\la u', U_0 \ra $ has hence codimension $r-1$ in $\la U\ra_{\F_{q^n}}$.
		
		\medskip
		
		If $d(r-1)  < m$, then there is at least a non-zero vector in $\la u', U_0 \ra \cap \la u', U_0 \ra^{\sigma} \cap \cdots \cap \la u', U_0 \ra^{\sigma^{d-1}} $ that is in $\la u^{\sigma^i},i=0,1,\ldots,n-1\ra \cap U$. Then this intersection is not the zero-vector and hence $u+u^{\sigma}+\ldots+u^{\sigma^{n-1}}$ defines a point of $L_U$.
		
		\medskip
		
		Suppose now that $r=2$ and $m \leq d$. If $m=d$ then there is at least a non-zero vector in $\la u', U_0 \ra \cap \la u', U_0 \ra^{\sigma} \cap \cdots \cap \la u', U_0 \ra^{\sigma^{d-2}} $, hence $\la u^{\sigma^i},i \not \equiv d-1 \pmod d, w_j, j \equiv d-1 \pmod d \ra \cap \la U \ra_{\F_{q^n}}\neq \{\mathbf{0}\}$ for some $w_j \in V_j$. By Corollary \ref{reducibility}, we get $\la u^{\sigma^i},i \not\equiv d-1 \pmod d, w_j, j \equiv d-1 \pmod d \ra \cap U_i \neq \{\mathbf{0}\}$ for some $i=0,1,\ldots,d-1$. Since in $\la u', U_0 \ra \cap \la u', U_0 \ra^{\sigma} \cap \cdots \cap \la u', U_0 \ra^{\sigma^{d-2}} $ we have a vector projected on $\displaystyle \sum_{i \not\equiv d-1 \mod d} u^{\sigma^i}$  from $\displaystyle \opl_{j \equiv d-1 \pmod d} V_j$, such a vector is either in $U_{d-1}$ or it does not belong to any $U_j$. In the latter case, $\la u^{\sigma^i},i \not\equiv d-1 \pmod d, w_j, j \equiv d-1 \pmod d \ra$ has a vector not belonging to any $U_j$ and a vector belonging to $U_i$, hence it contains a line of $\la U \ra_{\F_{q^n}}$, hence  such a line has a point on every hyperplane of  $\la u^{\sigma^i},i \not\equiv d-1 \pmod d, w_j, j \equiv d-1 \pmod d \ra$ and so a non-zero vector of every $U_i$. Therefore,  $\la u^{\sigma^i} : i \not\equiv d-1 \ra \cap \PG(U,\F_{q^n}) \neq \emptyset$, so $\la u^{\sigma^i} : i=0,1,\ldots,n-1  \ra \cap \PG(U,\F_{q^n}) \neq \emptyset$ and finally $\la u^{\sigma^i} : i=0,1,\ldots,n-1  \ra \cap \PG(U,\F_q) \neq \emptyset$.
		
		Suppose that $m<d$. Then  $\la u', U_0 \ra \cap \la u', U_0 \ra^{\sigma} \cap \cdots \cap \la u', U_0 \ra^{\sigma^{m-2}} \neq \{\mathbf{0}\}$ and since  $\la u', U_0 \ra $ is fixed by $\sigma^d$,  $\la u', U_0 \ra \cap \la u', U_0 \ra^{\sigma} \cap \cdots \cap \la u', U_0 \ra^{\sigma^{m-2}} \cap Fix(\sigma^d)\neq \{\mathbf{0}\}$, so there exists $v_j \in V_j$ with $j=m-1,m,\ldots, d-1$ such that
		
		\noindent$\la u^{\sigma^i},v_j, u^{\sigma^{d+i}},v_{j}^{\sigma^d},\ldots, u^{\sigma^{n-d+i}},v_{j}^{\sigma^{n-d}} \colon i=0,1,\ldots,m-2,\,j=m-1,m,\ldots, d-1 \ra \cap \la U \ra_{\F_{q^n}} \neq \{ \mathbf{0} \}$.

		By Corollary \ref{reducibility}, we get that
		
		\noindent$\la u^{\sigma^i},v_j, u^{\sigma^{d+i}},v_{j}^{\sigma^d},\ldots, u^{\sigma^{n-d+i}},v_{j}^{\sigma^{n-d}} \colon i=0,1,\ldots,m-2,\,j=m-1,m,\ldots, d-1 \ra \cap \la U_i \ra_{\F_{q^n}} \neq \{\mathbf{0}\}$ for some $i \in \{0,1,\ldots,d-1\}$. Again,
		
		\noindent$\la u^{\sigma^i},v_j, u^{\sigma^{d+i}},v_{j}^{\sigma^d},\ldots, u^{\sigma^{n-d+i}},v_{j}^{\sigma^{n-d}} \colon i=0,1,\ldots,m-2,j=m-1,m,\ldots, d-1 \ra\cap \la U_i \ra_{\F_{q^n}} $ is set-wise fixed by $\sigma^d$, so we can take a non-zero vector in
		
		\noindent$\la u^{\sigma^i},v_j, u^{\sigma^{d+i}},v_{j}^{\sigma^d},\ldots, u^{\sigma^{n-d+i}},v_{j}^{\sigma^{n-d}} \colon i=0,1,\ldots,m-2,j=m-1,m,\ldots, d-1 \ra \cap \la U_i \ra_{\F_{q^n}} \cap Fix (\sigma^d)$.

		Hence, by substituting $v_i$ by a multiple, we can take $u''$ such that
		
		\noindent$u''= \alpha_0u+\alpha_1u^{\sigma}+\cdots + \alpha_{m-2}u^{\sigma^{m-2}}+ \alpha_{m-1}v_{m-1}+\alpha_{m}v_m+\cdots +\alpha_{d-1}v_{d-1}+ \alpha_0^{\sigma^d}u^{\sigma^d}+ \alpha_1^{\sigma^d}u^{\sigma^{d+1}}+\cdots + \alpha_{d+m-2}^{\sigma^d}u^{\sigma^{d+m-2}}+ \alpha_{m-1}^{\sigma^d}v_{m-1}^{\sigma^d}+\alpha_m^{\sigma^d}v_m^{\sigma^d}+\cdots +\alpha_{d-1}^{\sigma^d}v_{d-1}^{\sigma^d}+ \cdots +\alpha_{0}^{\sigma^{n-d}} u^{\sigma^{n-d}}+ \alpha_{1}^{\sigma^{n-d+1}}u^{\sigma^{n-d+1}}+\cdots + \alpha_{m-2}^{\sigma^{n-d}}u^{\sigma^{n-d+m-2}}+ \alpha_{m-1}^{\sigma^{n-d}}v_{m-1}^{\sigma^ {n-d}}+\alpha_m^{\sigma^{n-d}}v_m^{\sigma^{n-d}}+\cdots +\alpha_{d-1}^{\sigma^{n-d}}v_{d-1}^{\sigma^{n-d}}\in \la U\ra_{\F_{q^n}}$ such that there is at least a $\alpha_j=0$ for some $j \in \{m-1,m,\ldots,d-1\}$.  Since $k$ from Lemma \ref{span} applied to $u''$ is at most $m<d \leq n$, we get that $\alpha_hv_h$ is a multiple of $u^{\sigma^j}$ $\forall \, h \in \{m-1,m,\ldots,d-1\}$. Therefore $\la u^{\sigma^i} : i=0,1,\ldots,n-1  \ra \cap \PG(U,\F_q) \neq \emptyset$.
	\end{proof}

	We can finally prove the general case.
	
	\begin{theorem}\label{final}
		Let $n\leq q$. The $\F_{q}$-subspaces $U$ and $\overline{U}\cap Fix(\sigma)$ define the same linear set in $\PG(r-1,q^n)$ for any $m$ and $r$.
	\end{theorem}
	\begin{proof}
		We prove the result by induction on $r$. For $r=2$, the result is proven in Theorem \ref{main}. Suppose that is true for $r-1$ and let $U$ such that $\la L_U \ra= \PG(r-1,q^n)$. If $\la L_U \cap H \ra$ is degenerate, i.e. does not span $H$, for every hyperplane  $H$ of $\PG(r-1,q^n)$, then  $\la L_U \ra \subsetneq \PG(r-1,q^n)$, a contradiction. Then there exists a hyperplane $H$ such that $\la L_U \cap H \ra= H$. The linear set $L_U\cap H$ is $L_W$, where $W=U\cap V'$ and $V'$ is the $\F_{q^n}$-vector space inducing $H$. Then $  W=U\cap V'$ in $V(rn,q^n)$ can be seen as $\la W \ra_{\F_{q^n}}= \la U \ra_{\F_{q^n}}\cap  V'\oplus V'^{\sigma} \oplus \cdots V'^{\sigma^{n-1}}= U_0+U_1+\cdots U_{d-1} \cap  V'\oplus V'^{\sigma} \oplus \cdots \oplus V'^{\sigma^{n-1}}=W_0+W_1 + \cdots + W_{d-1}$, with $W_i=\la W \ra_{\F_{q^n}}\cap \displaystyle\opl_{k \not\equiv i \pmod d} V_k$. We have assumed the Theorem true for $V'$, hence $L_W=L_{\la W\ra_{\F_{q^d}}}$. Since $\la L_{\la W\ra_{\F_{q^d}}}\ra =H$, $L_{\la W\ra_{\F_{q^d}}}$ contains at least a subgeometry $\cong \PG(r-2,q^d)$, hence $L_W$ is an $\F_q$-linear set containing every point of a $\PG(r-2,q^d)$, that is, there exists a subspace $W'$ of $W$ such that $L_{W'}=\PG(r-2,q^d)$. Let $w$ be the minimum weight of $L_{W'}$ and let $P$ be a point of weight $w$. Since the minimum weight of $L_{W'}$ is $w$, a subspace of codimension $w-1$ of $W'$ will determine the same set of points, hence $\dim_q W' \geq (r-2)d+w$. Let $Q$ be a point of $L_U$ not in $H$. On $\la P,Q \ra \cap U$ we also must have an $\F_q$-linear set containing the points on a line $\cong \PG(1,q^d)$, hence an $\F_q$-subspace of dimension at least $d+1$. Therefore, $\dim_q U \geq (r-2)d+w+d+1-w=(r-1)d+1 >(r-1)d$, hence we can apply Theorem \ref{main} and the result is proved.
	\end{proof}

	In the following, we include another proof of Theorem \ref{thm:main1} as an easy corollary of Theorem \ref{final}, but of course it is valid under the hypothesis $q\geq n$, which not required in Theorem \ref{thm:main1}.

	\begin{corollary}
		If $m$ is a multiple of $n$, then $U$ is an $\F_{q^d}$-vector space with $d$  the minimum weight.
	\end{corollary}
	\begin{proof}
		Let $m=an$, $1 \leq a  \leq r-1$.  The minimum size for a blocking set of  $\PG(r-1,q^n)$ with respect $(r-a-1)$-dimensional projective subspaces is $\frac{q^{n(a+1)}-1}{q^n-1}$, see \cite{BB66}. We have $|L_U|\leq \frac{q^m-1}{q-1}=\frac{q^{an}-1}{q-1}$ (see, e.g., \cite{OP2010}), and hence $|L_U|<\frac{q^{n(a+1)}-1}{q^n-1}$ and there exists an  $(r-a-1)$-dimensional projective subspaces $\pi$ of  $\PG(r-1,q^n)$ disjoint from $L_U$. Suppose that $L_U=L_{\la U \ra_{\F_{q^d}}}$ and $U$ is properly contained in $\la U \ra_{\F_{q^d}}$, then $\dim_q \la U \ra_{\F_{q^d}}=an +hd$, for some $h \geq 1$, $d\mid n$, $d \geq 2$. Then, by the Grassmann formula, the vector space inducing $\pi$ and $ \la U \ra_{\F_{q^d}}$ have non-zero  intersection, that is, $L_{\la U \ra_{\F_{q^d}}} \cap \pi \neq \emptyset$, a contradiction. Hence $U$ is an $\F_{q^d}$-vector subspace. If $d$ is smaller than the minimum weight, then we can iterate the process until we find a $d'$, $d < d' \mid n$ such that $U$ is an $\F_{q^{d'}}$-vectors space and $d'$ is the minimum weight.
	\end{proof}
	
	By the proof above, it also follows the dimension of $\la U \ra_{\F_{q^d}}$ over $\F_q$ is at most $\lceil \frac{m}{n} \rceil$.

	\begin{remark}
		\rm{We stress out that the minimum weight defines the field of linearity of $L_U$ only if $L_U$ is of maximum rank. For lower ranks, the field of linearity is defined by $d \mid n$ such that $\la U \ra_{\F_{q^n}}=U_0+U_1+\cdots + U_{d-1}$ and $U_0^{\sigma^d}=U_0$. Here, an example. Let $\mathrm{Tr}:\F_{q^3}\rightarrow \F_q$ the usual trace map and let $U=\{(x,y),x,y \in \F_{q^3}, \mathrm{Tr}(y)=0 \}$ be an  $\F_{q}$-subspace of $V(2,q^6)$, then $L_U$ is an $\F_q$-linear set of $\PG(1,q^6)$ of rank $5$. The maximum rank for an $\F_q$-linear set of $\PG(1,q^6)$ is $6$. If $(x,y) \in U$ and $y \neq 0$, then $(\lambda x, \lambda y) \in U$ if and only if $\lambda \in \F_{q^3}$ such that $\mathrm{Tr}(\lambda y)=0$, hence $\lambda$ belongs to an $\F_{q}$-vector space of dimension 2. If $y=0$, $\lambda (x,0) \in \forall\, \lambda \in \F_{q^3}$. Therefore, every point of $L_U$ has weight at least 2. In $\displaystyle\opl_{i=0}^5 V_i$, $U=\{(x,y,x^q,y^q,x^{q^2},-y-y^q,x,y,x^q,y^q,x^{q^2},-y-y^q) : x, y \in \F_{q^3}\}$. Then $\la U \ra _{\F_{q^6}}=\{(x_0,y_0,x_1,y_1,x_{2},-y_0-y_1,x_0,y_0,x_1,y_1,x_2,-y_0-y_1) : x_i,y_i \in \F_{q^6}\}$ and $U_0=\{(0,0,x_1,y_1,x_2,-y_1,0,0,x_1,y_1,x_2,-y_1) : x_i,y_i \in \F_{q^6}\}$ is fixed by $\sigma^3$.}

	\end{remark}
	
	\section{Consequences on the size of a linear set}
	\label{sec:4}
	
	In \cite{DBVdV}, it has been proved that the minimum size of a linear set of $\PG(1,q^n)$ of rank $k$ is $q^{k-1}+1$, provided that there is at least a point of weight $1$. Now we can rephrase the result:

	Let $L$ be a linear set of $\PG(1,q^n)$  such that $\F_q$ is the maximum field of linearity and such that its rank with respect $\F_q$ is $k$. Then $|L|\geq q^{k-1}+1$.

	Let $L=L_U$ be a linear set of $\PG(r-1,q^n)$ such that $\F_q$ is the maximum field of linearity and such that its \underline{rank with respect to $\F_q$ is $(r-2)n+k$ for some $k>0$}. Let $P$ be a point of weight $1$. By Grassmann's identity, we know that every line contains at least $q^k$ vectors of $U$.
	Put $\theta=(q^{n(r-1)}-1)/(q^n-1)$. By counting the vectors in each line through $P$ not defining the point $P$, we get
	\[
	\displaystyle\sum_{i=1}^{\theta}(q^{h_i}-q)=q^{(r-2)n+k}-q,
	\]
	where $h_i \geq k$ is the weight of the line $\ell_i$, $i=1,2,\ldots,\theta$, incident with $P$. Hence $\displaystyle\sum_{i=1}^{\theta}q^{h_i}=q^{(r-2)n+k}+q(\theta-1)$.
	By \cite{DBVdV}, on each line through $P$ we have at least $q^{h_i-1}+1$ points, hence
	\[
	|L_U| \geq \displaystyle\sum_{i=1}^{\theta}q^{h_i-1}+1 = q^{(r-2)n+k-1} +\theta.
	\]
	
	Let the \underline{$\F_q$-rank of $L_U$ be $k \leq (r-2)n$} and let $P$ be a point of weight $1$. If there is at least a line intersecting $L_U$ in an $\F_q$-subline, then by \cite[Theorem 1.4]{SP} lower bounds on the size of $L_U$  are known. Now suppose that this is not the case and every line through $P\in L_U$ is either tangent or meets $L_U$ in an $\F_q$-linear set of rank at least $3$, then the projection of $L_U$ from $P$ on a hyperplane $H$ not through $P$ is a linear set of $\PG(r-2,q^n)$ of rank $k-1 \leq (r-2)n-1$ such that every point has weight at least $2$, so it is equivalent to an $\F_{q^d}$-linear set for some divisor $1<d \mid n$. When $n$ is a prime then this is only possible if $n=d$ and hence the projection of $L_U$ is a projective subspace $H'$ of $H$. Then $\la L_U\ra=\la H',P\ra$ and $|L_U|\geq |H'|\,q^2+1$. Since $|L_U|\leq (q^{(r-2)n}-1)/(q-1)$, $H'$ is a proper subspace of $H$ and hence $\la L_U \ra \neq \PG(r-1,q^n)$. See the related  \cite[Open Problem (B)]{vdvn}.
	
	In this paper we proved that if $L_U$ has no points of weight one then it is linear over a larger field. One might wonder whether the absence of lines of weight $2$ yields the same conclusion. This is not the case since
	Example \ref{new} provides some $\F_q$-linear sets $L_U$, $\la L_U\ra = \PG(2,q^n)$, with maximum field of linearity $\F_q$ and such that there are no lines meeting $L_U$ in an $\F_q$-subline.
	We will need the following lemma.
	
	\begin{lemma}
		\label{nw}
		If $S\subseteq \PG(1,q^{mn})$, $q\geq mn$, is an $\F_{q^d}$-linear set of size $q^n+1$ then $d$ divides $n$.
	\end{lemma}
	\begin{proof}
		Assume that $d$ is maximal such that $d$ does not divide $n$ and $S=L_U$ for some $k$-dimensional $\F_{q^d}$-subspace $U$. Then the size of $L_U$ is at most $1+q^d+\ldots+q^{(k-1)d}$ and hence $(k-1)d \geq n+1$ (since $d$ does not divide $n$). Hence $k \geq (n+1)/d+1$. If $L_U$ has a point of weight $1$ then the size of $L_U$ is at least $q^{n+1}+1$, cf. \cite{DBVdV}, a contradiction. Thus $L_U$ does not contain points of weight $1$ and hence it is linear over an extension of $\F_{q^d}$, cf. Theorem \ref{thm:main},  contradicting the choice of $d$.
	\end{proof}

	\begin{example}
		\label{new}
		Put $n=4k+2$ for some positive integer $k$. Consider the $3$-dimensional $\F_{q^n}$-vector space $W=W_1\oplus W_2 \oplus W_3$, where $W_i$ is a one-dimensional $\F_{q^n}$-subspace for $i=1,2,3$ and $q\geq n$. Also, consider $U_1$ and $U_2$, two one-dimensional $\F_{q^{n/2}}$-subspaces contained in $W_1$ and $W_2$, respectively. Let $V$ denote an $(n/2+2)$-dimensional $\F_q$-subspace of $U_1 \oplus U_2$.
		
		Then $L_{U_1\oplus U_2}\cong\PG(1,q^{n/2})$ is contained in $\ell_0:=L_{W_1\oplus W_2}\cong\PG(1,q^n)$. The points of $L_{U_1\oplus U_2}$ have weight $n/2$. By Grassmann's Identity, $L_V=L_{U_1 \oplus U_2}$ and the points of $L_V$ have weight at least $2$. Note that there is at least one point $Q$ in $L_V$ of weight $2$ since otherwise we would have $|V|\geq (q^3-1)(q^{n/2}+1)+1$, a contradiction.
		
		Take a vector ${\bf v}\in W\setminus (W_1 \oplus W_2)$ and define $U=\la V, {\bf v}\ra_{\F_q}$. The linear set $L_U$ of $\PG(2,W)\cong \PG(2,q^n)$ is of rank $n/2+3$. The points of $L_U\setminus \ell_0$ have weight $1$ since otherwise $U$ would have dimension larger than $n/2+3$. If $P_i=\la \lambda_i {\bf v}+{\bf v_i}\ra$, $\lambda_i\in \F_q^*$, ${\bf v_i} \in V$, for $i=1,2$, are two distinct points of $L_U \setminus \ell_0$ then $\la P_1,P_2\ra \cap \ell_0 = \la -\lambda_2{\bf v_1}+\lambda_1{\bf v_2}\ra$ is a point of $L_U$. The line $\la P_1, P_2\ra$ meets $L_U$ in a linear set of size $q^k+1$, where $k\geq 2$ is the weight of the point $\la P_1,P_2 \ra \cap \ell_0 \in L_U$. On the other hand, $L_U$ is not a linear set over $\F_{q^d}$, $d>1$. Indeed, because of the $(q^2+1)$-secants of $L_U$ (for example the non-tangent lines through $Q$ different from $\ell_0$), the only possibility is $L_U=L_{U'}$ for some $\F_{q^2}$-subspace $U'$ of $W$. But then $L_U \cap \ell_0 = L_{U_1 \oplus U2}\cong\PG(1,q^{n/2})$ would be an $\F_{q^2}$-linear set which contradicts Lemma \ref{nw}.
	\end{example}

\section*{Acknowledgment}
	This work was supported by the Italian National Group for Algebraic and Geometric Structures and their Applications (GNSAGA--INdAM). The first author acknowledges the partial support of the National Research, Development and Innovation Office – NKFIH, grant no. K 124950.

	\bigskip
	\bigskip
	
	\noindent Bence Csajb\'ok\\
	Dipartimento di Meccanica, Matematica e Management,\\ Politecnico di Bari,\\ Via Orabona 4, 70125 Bari, Italy; \\
	{\em bence.csajbok@poliba.it}

	\bigskip
	
	\noindent Giuseppe Marino\\
	Dipartimento di Matematica e Applicazioni ``Renato Caccioppoli",\\
	Università degli Studi di Napoli Federico II,\\
	Via Cintia, Monte S.Angelo I-80126 Napoli, Italy\\
	{\em giuseppe.marino@unina.it}
	
	\bigskip
	\noindent
	Valentina Pepe\\
	Dipartimento di Scienze di Base ed Applicate per l’Ingegneria,\\
	``Sapienza" Università
	di Roma, \\
	Via Antonio Scarpa, 10, 00161 Roma, Italy\\
	{\em valentina.pepe@uniroma1.it}
	
\end{document}